\documentclass[10pt]{article}
\usepackage{mathrsfs}
\input epsf
\usepackage{amssymb}
\usepackage{amsmath}
\usepackage{amsfonts}
\input amssym.def
\newtheorem{theorem}{Theorem}[section]

\newtheorem{prop}[theorem]{Proposition}

\newenvironment{proof}{\noindent {\em Proof  }}{\hspace*{1cm}\hspace*{\fill}$\rule{1.2ex}{1.4ex}$\vspace{.15cm}}
\begin{document}
\title{\bf Well-posedness of the Fifth Order
Kadomtsev-Petviashvili I Equation in Anisotropic Sobolev Spaces with
Nonnegative Indices\footnote{This project was completed when the
first-named author visited Memorial University of Newfoundland under
the financial support from the NNSF of China No.10626008 as well as
the second-named author's NSERC (Canada) grant and Dean of Science
(MUN, Canada) Start-up fund.}}

\author{Junfeng Li\\
{\it School of Mathematical Sciences}\\
{\it Laboratory of Math and Complex Systems, Ministry of
Education}\\
{\it Beijing Normal University, Beijing 100875, P. R. China}\\
{\it Email: junfli@yahoo.com.cn}\\
\hspace{0.5cm}\\
Jie Xiao\\
{\it Department of Mathematics and Statistics}\\
{\it Memorial University of Newfoundland, St John's, NL AIC 5S7, Canada}\\
{\it Email: jxiao@math.mun.ca}}

\date{}

\maketitle

\begin{abstract} In this paper we establish the local and
global well-posedness of the real valued fifth order
Kadomstev-Petviashvili I equation in the anisotropic Sobolev spaces
with nonnegative indices. In particular, our local well-posedness
improves Saut-Tzvetkov's one and our global well-posedness gives an
affirmative answer to Saut-Tzvetkov's $L^2$-data conjecture.
\end{abstract}
\smallskip


\section{Introduction}

In their J. Math. Pures Appl. (2000) paper on the initial value
problem (IVP) of the real valued fifth order Kadomtsev-Petviashvili
I (KP-I) equation (for $(\alpha,t,x,y)\in\mathbb R^4$):
\begin{eqnarray}\label{5KP}
\left\{\begin{array}{ll}
\partial_tu+\alpha\partial_x^3u+\partial_x^5u+\partial_x^{-1}\partial_y^2u+u\partial_xu=0, \\
u(0,x,y)=\phi(x,y),
\end{array}
\right.
\end{eqnarray}
J.C. Saut and N. Tzvzetkov obtained the following result (cf. \cite[Theorems 1 \& 2]{SauTzv2}):

\medskip

\noindent{\bf Saut-Tzvzetkov's Theorem}\, {\it {\rm(i)} The IVP
(\ref{5KP}) is locally well-posed for initial data $\phi$ satisfying
\begin{equation}\label{eq1new}
\|\phi\|_{L^2(\mathbb
R^2)}+\big\||-i\partial_x|^s\phi\big\|_{L^2(\mathbb
R^2)}+\big\||-i\partial_y|^k\phi\big\|_{L^2(\mathbb R^2)}<\infty\
\hbox{with}\ s-1,\ k\ge 0;\ \
\frac{\hat{\phi}(\xi,\eta)}{|\xi|}\in\mathcal{S}'(\mathbb R^2).
\end{equation}

{\rm (ii)} The IVP (\ref{5KP}) is globaly well-posed for initial
data $\phi$ satisfying
\begin{equation}\label{eq2new}
\|\phi\|_{L^2(\mathbb R^2)}<\infty;\quad
\frac12\int_{\mathbb R^2}
|\partial^2_x\phi|^2+\frac{\alpha}{2}\int_{\mathbb R^2}
|\partial_x\phi|^2+\frac12\int_{\mathbb R^2}
|\partial^{-1}_x\partial_y\phi|^2-\frac16\int_{\mathbb
R^2}\phi^3<\infty.
\end{equation}
Here and henceforth, $|-i\partial_x|^s$ and  $|-i\partial_y|^s$ are
defined via the Fourier transform:
$$
\widehat{|-i\partial_x|^s\phi}(\xi,\eta)=|\xi|^s\hat{\phi}(\xi,\eta)\quad\hbox{and}\quad
\widehat{|-i\partial_y|^s\phi}(\xi,\eta)=|\eta|^s\hat{\phi}(\xi,\eta).
$$
}

Since they simultaneously found in \cite[Theorem 3]{SauTzv2} that
the condition
\begin{equation}\label{eq3new}
\|\phi\|_{L^2(\mathbb R^2)}<\infty;\quad
|\xi|^{-1}\hat{\phi}(\xi,\eta)\in\mathcal{S}'(\mathbb R^2)
\end{equation}
ensures the gobal well-posedness for the real valued fifth order Kadomtsev-Petviashvili II (KP-II) equation (for $(\alpha,t,x,y)\in\mathbb R^4$):
\begin{eqnarray}\label{5KP2}
\left\{\begin{array}{ll}
\partial_tu+\alpha\partial_x^3u-\partial_x^5u+\partial_x^{-1}\partial_y^2u+u\partial_xu=0, \\
u(0,x,y)=\phi(x,y),
\end{array}
\right.
\end{eqnarray}
they made immediately a conjecture in \cite[Remarks, p. 310]{SauTzv2} which is now reformulated in the following form:
\medskip

\noindent{\bf Saut-Tzvzetkov's $L^2$-data Conjecture}\, {\it The IVP
(\ref{5KP}) is globally well-posed for initial data $\phi$
satisfying (\ref{eq3new})}.
\medskip

In the above and below, as ``local well-posedness" we refer to finding
a Banach space $(X,\|\cdot\|_X)$ -- when the initial data $\phi\in
X$ there exists a time $T$ depending on $\|\phi\|_X$ such that
(\ref{5KP}) has a unique solution $u$ in $C([-T,T];X)\cap Y$ (where
$Y$ is one of the Bourgain spaces defined in Section 2) and $u$
depends continuously on $\phi$ (in some reasonable topology). If
this existing time $T$ can be extended to the positive infinity,
then ``local well-posedness" is said to be ``global well-posedness". Of
course, the choice of a Banach space relies upon the
boundedness of the fundamental solution to the corresponding
homogenous equation or the conservation law for equation itself.

In our current paper, we settle this conjecture through improving the above-cited Saut-Tzvzetkov's theorem.
More precisely, we have the following:

\begin{theorem}\label{mainth} The IVP (\ref{5KP}) is not only locally but also globally well-posed for initial data $\phi$ satisfying
\begin{equation}\label{eq4new}
\phi\in H^{s_1,s_2}(\mathbb R^2)\ \hbox{with}\ s_1,s_2\ge 0;\quad
|\xi|^{-1}\hat{\phi}(\xi,\eta)\in\mathcal{S}'(\mathbb R^2).
\end{equation}
\end{theorem}

Here and henceafter, the symbol
\begin{equation*}
H^{s_1,s_2}(\mathbb{R}^2)=\Big\{f\in
\mathcal{S}^\prime(\mathbb{R}^2):\,\|f\|_{H^{s_1,s_2}(\mathbb{R}^2)}=\big\|(1+|\xi|^2)^{\frac{s_1}{2}}(1+|\eta|^2)^{\frac{s_2}{2}}\hat{f}(\xi,\eta)\big\|_{L^2(\mathbb
R^2)}<\infty\Big\}
\end{equation*}
stands for the anisotropic Sobolev space with nonnegative indices
$s_1,s_2\in [0,\infty)$. Obviously, if $s_1=s_2=0$ then
$H^{s_1,s_2}(\mathbb R^2)=L^2(\mathbb R^2)$ and hence (\ref{eq4new})
goes back to (\ref{eq3new}) which may be regarded as the appropriate
constraint on the initial data $\phi$ deriving the global
well-posednedness of the IVP for the fifth order KP-I equation. And
yet the understanding of Theorem \ref{mainth} is not deep enough
without making three more observations below:

\medskip

\noindent$\bullet$\quad {\bf Observation 1}\quad The classification
of the fifth order KP equations is determined by the dispersive
function:
\begin{equation}\label{KPsign}
\omega(\xi,\,\mu)=\pm\xi^5-\alpha\xi^3+\frac{\mu^2}{\xi},
\end{equation}
where the signs $\pm$ in (\ref{KPsign}) produce the fifth order KP-I
and KP-II equations respectively. The forthcoming estimates play an
important role in the analysis of the fifth order KP equations --
for the fifth order KP-I equation, we have
\begin{equation}\label{5KP-1sign}
|\xi|^2>|\alpha|\Rightarrow|\nabla\omega(\xi,\mu)|=\Big|\Big(5\xi^4+3\alpha\xi^2-\frac{\mu^2}{\xi^2},\,2\frac{\mu}{\xi}\Big)\Big|\gtrsim
|\xi|^2;
\end{equation}
and for the fifth order KP-II equation, we have
\begin{equation}\label{5KP-2sign}
 |\xi|^2>|\alpha|\Rightarrow |\nabla\omega(\xi,\mu)|=\Big|\Big(5\xi^4+3\alpha\xi^2+\frac{\mu^2}{\xi^2},\,2\frac{\mu}{\xi}\Big)\Big|\gtrsim
|\xi|^4.
\end{equation}
By (\ref{5KP-2sign}), we can get more smooth effect estimates than
by $(\ref{5KP-1sign})$. These imply that we can get a well-posedness (in other words, a lower regularity) for the fifth order KP-II equation better than that for the fifth order
KP-I equation. Another crucial concept is the resonance function:
\begin{equation}\label{5KP-Iresonance}
\begin{split}
&R(\xi_1,\xi_2,\mu_1,\mu_2)\\
&=\omega(\xi_1+\xi_2,\mu_1+\mu_2)-\omega(\xi_1,\mu_1)-\omega(\xi_2,\mu_2)\\
&=\frac{\xi_1\xi_2}{(\xi_1+\xi_2)}\left((\xi_1+\xi_2)^2\Big[5(\xi_1^2+\xi_1\xi_2+\xi_2^2)-3\alpha\Big]\mp
\Big(\frac{\mu_1}{\xi_1}-\frac{\mu_2}{\xi_2}\Big)^2\right).
\end{split}
\end{equation}
Evidently, the fifth order KP-II equation (corresponding to ``+" in
(\ref{5KP-Iresonance})) always enjoys
\begin{equation}\label{eqlast}
|R(\xi_1,\xi_2,\mu_1,\mu_2)|\gtrsim \big(\max\{|\xi_1|,|\xi_2|,|\xi_1+\xi_2|\}\big)^4\min\{|\xi_1|,|\xi_2|,|\xi_1+\xi_2|\}.
\end{equation}
Nevertheless, this last inequality (\ref{eqlast}) is no longer true for the fifth order KP-I equation.

In the foregoing and following the notation $A\lesssim B$ (i.e.,
$B\gtrsim A$) means: there exists a constant $C>0$ independent of
$A$ and $B$ such that $A\leq CB.$ In addition, if there
exist two positive constants $c$ and $C$ such that $10^{-3}<c<C<10^3$ and $cA\leq
B\leq CB$ then the notation $A\sim B$ will be used.

\medskip

\noindent$\bullet$\quad{\bf Observation 2}\quad Perhaps it worths
pointing out that the well-posedness of the fifth order KP-II
equation is relatively easier to establish but also its result is
much better than that of the fifth order KP-I equation. Although the
study of the well-posedness for the fifth order KP-II equation
(without the third order partial derivative term) usually focuses on
the critical cases (which means $s_1+2s_2=-2$ by a scaling
argument), in \cite{SauTzv1} Saut and Tzvetkov only obtained the
local well-posedness for the fifth order KP-II equation in the
anisotropic Sobolev space $H^{s_1,s_2}(\mathbb R^2)$ with
$s_1>-\frac{1}{4},s_2\geq 0$ with a modification of the low
frequency, and furthermore in \cite{SauTzv2} they removed this
modification and obtained the global well-posedness in $L^2(\mathbb
R^2)$. On the other hand, in \cite{IsLM} Isaza-L\'opez-Mej\'ia
established the local well-posedness for $H^{s_1,s_2}(\mathbb R^2)$
with $s_1>-\frac{5}{4},s_2\geq0$ and the global well-posedness for
$H^{s_1,s_2}(\mathbb R^2)$ with $s_1>-\frac{4}{7},s_2\geq0$. More
recently, Hadac \cite{Had} also gained the same local well-posedness
in a broader context. Meanwhile in the fifth order KP-I equation
case, the attention is mainly paid on those spaces possessing
conservation law such as $L^2(\mathbb R^2)$ and the energy space
$$
E^1(\mathbb R^2)=\Big\{f\in L^2(\mathbb R^2):\quad
\big\|(1+|\xi|^2+|\xi|^{-1}|\mu|)\hat{f}(\xi,\mu)\big\|_{L^2(\mathbb
R^2)}<\infty\Big\}.
$$
To obtain the local well-posedness of KP-I in $E^1(\mathbb R^2)$, in
\cite{SauTzv2}, besides the above-mentioned results Saut and Tzvekov
also got the local well-posedness in $\tilde{H}^{s,k}(\mathbb R^2)$
with $s-1,k\ge 0$.
$$
\tilde{H}^{s,k}(\mathbb R^2)=\Big\{f\in
L^2(\mathbb{R}^2):\quad\big\|(1+|\xi|^s+|\xi|^{-1}|\eta|^k)\hat{f}(\xi,\eta)\big\|_{L^2(\mathbb
R^2)}<\infty\Big\}.
$$ For the energy case $\tilde{H}^{2,1}(\mathbb R^2)=E^1(\mathbb R^2)$, they obtained the global well-posedness of (\ref{5KP}). In \cite{IoKe},
Ionescu and Kenig got the global well-posedness for the fifth order
periodic KP-I equation (without the third order dispersive term) in
the standard energy space $E^1(\mathbb R^2)$. Recently, in
\cite{CLM1} Chen-Li-Miao obtained the local well-posedness in
$$
E^s(\mathbb R^2)=\Big\{f\in L^2(\mathbb R^2):\quad
\|(1+|\xi|^2+|\xi|^{-1}|\mu|)^s\hat{f}(\xi,\mu)\|_{L^2(\mathbb
R^2)}<\infty\Big\},\quad 0<s\leq1.
$$

\medskip

\noindent$\bullet$\quad{\bf Observation 3}\quad The well-posedness
for the IVP of the third order KP equations in $\mathbb R^3$:
\begin{eqnarray}\label{5KP3}
\left\{\begin{array}{ll}
\partial_tu\mp\partial_x^3u+\partial_x^{-1}\partial_y^2u+u\partial_xu=0, \\
u(0,x,y)=\phi(x,y),
\end{array}
\right.
\end{eqnarray}
in which the sign $\mp$ give the third order KP-I and KP-II
equations respectively, is an important background material of the
investigation of the well-posedness for the fifth order KP
equations. Molinet, Saut and Tzvetkov showed in
\cite{MoSaTz1,MoSaTz2} that, for the third order KP-I equation one
cannot obtain the local well-posedness in any type of nonisotropic
$L^2$-based Sobolev space or in the energy space using Picard's
iteration -- see also \cite{KoTz}; while I\'orio and Nunes
\cite{IoNu} applied a compactness method to deduce the local
well-posednes for the third KP-I equation with data being in the
normal Sobolev space $H^s(\mathbb{R}^2),\,s>2$ and obeying a
``zero-mass" condition. On the other hand, the global well-posedness
for the third order KP-I equation was discussed by using the
classical energy method in \cite{Ke} where Kenig established the
global well-posedness in
$$
\Big\{f\in L^2(\mathbb{R}^2):\quad\|f\|_{L^2(\mathbb
R^2)}+\|\partial_x^{-1}\partial_y f\|_{L^2(\mathbb R^2)}
+\|\partial_x^2 f\|_{L^2(\mathbb
R^2)}+\|\partial_x^{-2}\partial_y^2f\|_{L^2(\mathbb
R^2)}<\infty\Big\}.
$$
As far as we know, the best well-posed result on the third order
KP-I equation is due to Ionescu, Kenig and Tataru \cite{IoKeT} which
gives the global well-posedness for the third order KP-I equation in
the energy space
$$
\Big\{f\in L^2(\mathbb R^2):\quad\|f\|_{L^2\mathbb
R^2}+\|\partial_x^{-1}\partial_yf\|_{L^2(\mathbb
R^2)}+\|\partial_xf\|_{L^2(\mathbb{R}^2)}<\infty\Big\}.
$$
Relatively speaking, the results on the third order KP-II equation
are nearly perfect. In \cite{Bou}, Bourgain proved the global
well-posedness of the third order KP-II equation in $L^2(\mathbb
R^2)$ -- the assertion was then extended by Takaoka and Tzvetkov
\cite{TakTzv} and Isaza-Mej\'ia \cite{IsMe} from $L^2(\mathbb R^2)$
to $H^{s_1,s_2}(\mathbb R^2)$ with $s_1>-\frac{1}{3},\, s_2\geq 0$.
In \cite{Tak}, Takaoka obtained the local well-posedness for the
third order KP-II equation in $H^{s_1,s_2}(\mathbb R^2)$ with
$s_1>-\frac{1}{2},\,s_2=0$ under an additional low frequency
condition $|-i\partial_x|^{-\frac{1}{2}+\varepsilon}\phi\in
L^2(\mathbb R^2)$ which was removed successfully in Hadac's recent
paper \cite{Had}. These results are very close to the critical index
$s_1+2s_2=-\frac{1}{2}$ which follows from the scaling argument.

The rest of this paper is devoted to an argument for Theorem
\ref{mainth}. In Section 2 we collect some useful and basically known linear estimates
for the fifth order KP-I equation. In Section 3 we present the
necessary and crucial bilinear estimates in order to set up the local
(and hence global) well-posedness -- this part is partially motivated by \cite{SauTzv2} though --
the main difference between their treatment and ours is how to
dispose the ``high-high interaction" -- their method exhausts no
geometric structure of the resonant set of the fifth order KP-I
equation while ours does fairly enough. In Section 4 we complete the
argument through applying the facts verified in Sections 2 and 3 and
Picard's iteration principle to the integral equation corresponding
to (\ref{5KP}).

\section{Linear Estimates}
We begin with the IVP of linear fifth order KP-I equation:
\begin{eqnarray}\label{L5KP}
\left\{\begin{array}{ll}
\partial_tu+\alpha\partial_x^3u+\partial_x^5u+\partial_x^{-1}\partial_y^2u=0, \\
u(0,x,y)=\phi(x,y).
\end{array}\right.
\end{eqnarray}
By the Fourier transform $\widehat{(\cdot)}$, the solution of
(\ref{L5KP}) can be defined as
$$
u(t)(x,y)=\big(S(t)\phi\big)(x,y)=\int_{\mathbb{R}^2}e^{i(x\xi+y\mu+t\omega(\xi,\mu))}\widehat{\phi}(\xi,\mu)d\xi
d\mu.$$ By Duhamel's formula, (\ref{5KP}) can be reduced to the
integral representation below:
\begin{eqnarray}\label{eq:1}
u(t)=S(t)\phi-\frac 12 \int_0^tS(t-t')\partial_x(u^2(t'))dt'.
\end{eqnarray}
So, in order to get the locall well-posedenss we will apply a Picard
fixed point argument in a suitable function space to the following
integral equation:
\begin{eqnarray}\label{eq:2}
u(t)=\psi(t)S(t)\phi-\frac{\psi_T(t)}{2}
\int_0^tS(t-t')\partial_x(u^2(t'))dt',
\end{eqnarray}
where $t$ belongs to ${\Bbb R}$, $\psi$ is a time cut-off function
satisfying
$$
\psi\in C_0^\infty({\Bbb R});\ \ {\rm supp}\,\psi\subset [-2,\,2];\ \
\psi=1\ \ {\rm on}\ \ [-1,\,1],
$$
and $\psi_T(\cdot)$ represents $\psi(\cdot/T)$ for a given time
$T\in (0,1)$. Consequently, we need to define an appropriate
Bourgain type space, which is associated with the fifth order KP-I
equation. To this end, for $s_1,s_2\ge 0$ and $b\in \mathbb R$ the
notation $X_b^{s_1,s_2}$ is used as the Bourgain space with norm:

$$
\|u\|_{X^{s_1,s_2}_b}=\|<\tau-\omega(\xi,\mu)>^b<\xi>^{s_1}<\mu>^{s_2}\hat{u}(\tau,\xi,\mu)\|_{L^2(\mathbb
R^3)},
$$
where $<\cdot>$ stands for $(1+|\cdot|^2)^{\frac{1}{2}}\sim
1+|\cdot|$. Furthermore, for an interval $I\subset\mathbb{R}$ the
localized Bourgain space $X^{s_1,s_2}_b(I)$ can be defined via
requiring
$$
\|u\|_{X^{s_1,s_2}_b(I)}=\inf_{w\in
X^{s_1,s_2}_b}\big\{\|w\|_{X^{s_1,s_2}_b}:\quad w(t)=u(t)\quad \text{on\quad
interval}\quad I\big\}.
$$

The following two results are known.

\begin{prop}\cite{SauTzv2}\label{prop1} If
$$
T\in (0,\infty);\, s_1,s_2\geq 0;\,\, -\frac{1}{2}<b'\leq 0\leq
b\leq b'+1,
$$
then
\begin{equation}\label{homo}
\|\psi
S(t)\phi\|_{X^{s_1,s_2}_b}\lesssim\|\phi\|_{H^{s_1,s_2}(\mathbb
R^2)}.
\end{equation}
\begin{equation}\label{inhomo}
\Big\|\psi
(t/T)\int_0^tS(t-t')h(t')dt'\Big\|_{X^{s_1,s_2}_b}\lesssim
T^{1-b+b'}\|h\|_{X^{s_1,s_2}_{b'}}.
\end{equation}
for any $\|h\|_{X^{s_1,s_2}_{b'}}<\infty.$
\end{prop}
\begin{prop}\cite{BAS}\label{prop2} If $r\in [2,\infty)$, then there
exists a constant $c>0$ independent of $T\in (0,1)$ such that
\begin{equation}\label{Strichartz}
\big\||-i\partial_x|^{\frac12-\frac1r}\big(S(t)\phi\big)(x,y)\big\|_{L^{\frac{2r}{r-2}}_TL^r(\mathbb
R^2)}\leq c\|\phi\|_{L^2(\mathbb R^2)},
\end{equation}
\end{prop}
where
$$
\|f\|_{L^{\frac{2r}{r-2}}_TL^r(\mathbb R^2)}=\left(\int_{-T}^T\left(\int_{\mathbb R^2}|f(x,y,t)|^rdxdy\right)^{\frac{2}{r-2}}dt\right)^{\frac{r-2}{2r}}.
$$

To reach our bilinear inequalities in Section 3, we will use
$(\cdot)^\vee$ for the inverse Fourier transform, and take the
dyadic decomposed Strichartz estimates below into account.

\begin{prop}\label{dydicstri} Let $\eta$ be a bump function with
compact support in $[-2,2]\subset\mathbb R$ and $\eta=1$ on
$(-1,1)\subset\mathbb R$. For each integer $j\ge 1$ set
$\eta_j(x)=\eta(2^{-j}x)-\eta(2^{1-j}x)$, $\eta_0(x)=\eta(x)$,
$\eta_j(\xi,\mu,\tau)=\eta_j(\tau-\omega(\xi,\mu))$, and
$f_j(\xi,\mu,\tau)=(\eta_j(\xi,\mu,\tau)|\hat{f}|(\xi,\mu,\tau))^\vee
$ for any given $f\in L^2(\mathbb R^3)$. Then for given $r\in
[2,\infty)$ and any $T\in (0,1)$ we have
\begin{equation}\label{eq:3}
\big\||-i\partial_x|^{\frac12-\frac1r}f_j\big\|_{L^{\frac{2r}{r-2}}_T
L^r(\mathbb R^2)}\lesssim 2^{\frac{j}{2}}\|f_j\|_{L^2(\mathbb R^3)}.
\end{equation}
In particular,
\begin{equation}\label{stri-4}
\big\||-i\partial_x|^{\frac{1}{4}}f_j\big\|_{L^4_TL^4(\mathbb
R^2)}\lesssim 2^{\frac{j}{2}}\|f_j\|_{L^2(\mathbb R^3)}.
\end{equation}
\end{prop}

\begin{proof}: Note first that
$$
f_j(x,y,t)=\int_{\mathbb{R}^3}e^{i(x\xi+y\mu+t\tau)}|\hat{f}|\eta_j(\xi,\mu,\tau)d\xi d\mu d\tau.
$$
So, changing variables and using
$\widehat{f_\lambda}(\xi,\mu)=|\hat{f}|(\xi,\mu,\lambda+\omega)$ we
can write
\begin{equation*}
\begin{split}
f_j(x,y,t)&=\int_{\mathbb{R}^3}e^{i(x\xi+y\mu+t(\lambda+\omega))}|\hat{f}|(\xi,\mu,\lambda+\omega)
\eta_j(\lambda)d\xi d\mu d\lambda \\
 &=\int_{\mathbb{R}}e^{it\lambda}\eta_j(\lambda)
\Big[\int_{\mathbb{R}^2}e^{i(x\xi+y\mu+t\omega)}|\hat{f}|(\xi,\mu,\lambda+\omega)
d\xi d\mu \Big]d\lambda \\
&=\int_{\mathbb{R}} e^{it\lambda}\eta_j(\lambda)
S(t)f_\lambda(x,y)d\lambda.
\end{split}
\end{equation*}
Now the estimate (\ref{eq:3}) follows from Minkowski's inequality, the
Strichartz estimate (\ref{Strichartz}) and the Cauchy-Schwarz
inequality.
\end{proof}

The following well-known elementary inequalities are also useful --
see for example \cite[Proposition 2.2]{SauTzv2}.

\begin{prop} Let $\gamma>1$. Then
\begin{equation}\label{2.1}
\int_{\mathbb
R}\frac{dt}{<t>^\gamma<t-a>^\gamma}\lesssim<a>^{-\gamma}
\end{equation}
and
\begin{equation}\label{2.2}
\int_{\mathbb
R}\frac{dt}{<t>^\gamma|t-a|^{\frac{1}{2}}}\lesssim<a>^{-\frac{1}{2}}
\end{equation}
hold for any $a\in\mathbb R$
\end{prop}

 \hspace{5mm}

\section{Bilinear Estimates}

Although there were many works on the so-called bilinear estimates,
we have found that the Kenig-Ponce-Vega's bilinear estimation
approach introduced in \cite{KPV} is quite suitable for our purpose.
With the convention: when $a\in\mathbb R$ the number $a\pm$ equals
$a\pm\epsilon$ for arbitrarily small number $\epsilon>0$, we can
state our bilinear estimate as follows.

\begin{theorem}\label{biestimate} If $s_1,s_2\geq 0$ and functions $u,v$ have compact time support on
$[-T,T]$ with $0<T<1$, then
\begin{equation}\label{bilinear}\|\partial_x(uv)\|_{X^{s_1,s_2}_{-\frac{1}{2}+}}\lesssim
\|u\|_{X^{s_1,s_2}_{\frac{1}{2}+}}\|v\|_{X^{s_1,s_2}_{\frac{1}{2}+}}.
\end{equation}
\end{theorem}

\begin{proof} In what follows, we derive (\ref{bilinear}) using the duality;
that is, we are required to dominate the integral
\begin{equation}\label{eq:3.1}
\begin{split}
\int_{A^\ast}\frac{|\xi|<\xi>^{s_1}<\mu>^{s_2}}{<\tau-\omega(\xi,\mu)>^{\frac{1}{2}-}}g(\xi,\mu,\tau)
|\hat{u}|(\xi_1,\mu_1,\tau_1)|\hat{v}|(\xi_2,\mu_2,\tau_2)d\xi_1d\mu_1d\tau_1d\xi_2d\mu_2d\tau_2,
\end{split}
\end{equation}
where $g\ge 0$, $\|g\|_{L^2(\mathbb R^2)}\leq 1$ and
$$
A^\ast=\big\{(\xi_1,\mu_1,\tau_1,\xi_2,\mu_2,\tau_2)\in\mathbb R^6:\,\,
\xi_1+\xi_2=\xi,\,\mu_1+\mu_2=\mu,\,\tau_1+\tau_2=\tau\big\}.
$$
Let
$$
\sigma=\tau-\omega(\xi,\mu);\,\,\sigma_1=\tau_1-\omega(\xi_1,\mu_1);\,\,\sigma_2=\tau_2-\sigma(\xi_2,\mu_2).
$$
Define two functions below:
$$
f_1(\xi_1,\mu_1,\tau_1)=<\xi_1>^{s_1}<\mu_1>^{s_2}<\sigma_1>^{\frac{1}{2}+}|\hat{u}(\xi_1,\mu_1,\tau_1)|
$$
and
$$
f_2(\xi_2,\mu_2,\tau_2)=<\xi_2>^{s_1}<\mu_2>^{s_2}<\sigma_2>^{\frac{1}{2}+}|\hat{v}(\xi_2,\mu_2,\tau_2)|.
$$
Then we need to bound the integral
\begin{equation}\label{eq:3.2}
\begin{split}
\int_{A^\ast} K(\xi_1,\mu_1,\tau_1,\xi_2,\mu_2,\tau_2)g(\xi,\mu,\tau)
f_1(\xi_1,\mu_1,\tau_1)f_2(\xi_2,\mu_2,\tau_2)d\xi_1d\mu_1d\tau_1d\xi_2d\mu_2d\tau_2
\end{split}
\end{equation}
from above by using a constant multiple of $\|f_1\|_{L^2(\mathbb
R^3)} \|f_2\|_{L^2(\mathbb R^3)}$. Here
\begin{eqnarray*}
K(\xi_1,\mu_1,\tau_1,\xi_2,\mu_2,\tau_2)&=&\left(\frac{|\xi_1+\xi_2|}
{<\sigma>^{\frac{1}{2}-}<\sigma_1>^{\frac{1}{2}+}<\sigma_2>^{\frac{1}{2}+}}\right)\\
&\quad\quad\times&
\left(\frac{<\xi_1+\xi_2>^{s_1}}{<\xi_1>^{s_1}<\xi_2>^{s_1}}\right)
\left(\frac{<\mu_1+\mu_2>^{s_2}}{<\mu_1>^{s_2}<\mu_2>^{s_2}}\right).
\end{eqnarray*}
It is clear that for $s_1,s_2\geq 0$ we always have
$$K(\xi_1,\mu_1,\tau_1,\xi_2,\mu_2,\tau_2)\lesssim
\frac{|\xi_1+\xi_2|}{<\sigma>^{\frac{1}{2}-}<\sigma_1>^{\frac{1}{2}+}<\sigma_2>^{\frac{1}{2}+}}.$$

Keeping a further assumption $|\xi_1|\geq|\xi_2|$ (which follows from
symmetry) in mind, we are about to fully control
the integral in (\ref{eq:3.2}) through handling two situations.

\medskip

\noindent$\bullet$\quad{\bf Situation 1 -- Low Frequency}\quad
$|\xi_1+\xi_2|\lesssim\max\{10,|\alpha|\}.$

\medskip

$\circ$\quad{\it High+High$\rightarrow$Low}\quad
$|\xi_1|,|\xi_2|\gtrsim\max\{10,|\alpha|\}$.\quad We first deduce a
dyadic decomposition. Employing $\eta_j$ in Proposition
\ref{dydicstri}, we have $\sum_{j\geq 0}\eta_j=1$, and consequently
(\ref{eq:3.2}) can be bounded from above by a constant multiple of
\begin{equation}\label{eq:3.4}
\sum_{j\geq0}2^{-j(\frac{1}{2}-)}\int_{A^\ast}\eta_j(\sigma)g(\xi,\mu,\tau)
\left(\frac{f_1(\xi_1,\mu_1,\tau_1)}{<\sigma_1>^{\frac{1}{2}+}}\right)
\left(\frac{f_2(\xi_2,\mu_2,\tau_2)}{<\sigma_2>^{\frac{1}{2}+}}\right)d\xi_1d\mu_1d\tau_1d\xi_2d\mu_2d\tau_2.
\end{equation}
We may assume that for each natural number $j$,
$$
G_j(x,y,t)=\mathcal{F}^{-1}\Big(\eta_j(\sigma)g(\xi,\mu,\tau)\Big)(x,y,t),
$$
has support compact in the interval $[-T,T]$ whenever it acts as a
time-dependent function, where $\mathcal F^{-1}$ also denotes the
inverse Fourier transform. In fact, if we consider the following
functions generated by $\mathcal F^{-1}$:
$$
F_l(x,y,t)=\mathcal{F}^{-1}\Bigg(\frac{f_l(\xi_l,\mu_l,\tau_k)}{<\sigma_l>^{\frac{1}{2}+}}\Bigg)(x,y,t)\quad\text{for}\quad l=1,2,
$$
then the integral in (\ref{eq:3.4}) can be written as an $L^2$ inner product $\langle G_j,F_1F_2\rangle$. Since $u$ and $v$ (acting as time-dependent functions) have compact support in $[-T,T]$, so does $F_1F_2$. As a result, the inner product $\langle G_j,F_1F_2\rangle$ can be restricted on the interval $[-T,T]$, namely, we may assume that $G_j$ has the same compact support (with respect to time) as $F_1F_2$'s. Now, an application of (\ref{stri-4}) yields that the sum in (\ref{eq:3.4}) is bounded by a constant multiple of
\begin{equation*}
\begin{split}
&\sum_{j\geq0}2^{-j(\frac{1}{2}-)}\langle G_j,F_1F_2\rangle\\
&\lesssim\sum_{j,j_1,j_2\geq0}\Big(2^{-j(\frac{1}{2}-)}2^{-j_1(\frac{1}{2}+)}2^{-j_2(\frac{1}{2}+)}\\
&\quad\times
\big\||-i\partial_x|^{\frac{1}{4}}(\eta_{j_1}(\sigma_1)f_1)^\vee\big\|_{L^4_TL^4(\mathbb
R^2)}
\big\||-i\partial_x|^{\frac{1}{4}}(\eta_{j_2}(\sigma_2)f_2)^\vee\big\|_{L^4_TL^4(\mathbb
R^2)}
\|\eta_{j}(\sigma)g\|_{L^2(\mathbb R^3)}\Big)\\
&\lesssim\sum_{j,j_1,j_2\geq0}\Big(2^{-j(\frac{1}{2}-)}2^{-j_1[(\frac{1}{2}+)-\frac{1}{2}]}2^{-j_2[(\frac{1}{2}+)-\frac{1}{2}]}\\
&\quad\times \|\eta_{j_1}(\sigma_1)f_1\|_{L^2(\mathbb R^3)}
\|\eta_{j_2}(\sigma_2)f_2\|_{L^2(\mathbb R^3)}
\|\eta_{j}(\sigma)g\|_{L^2(\mathbb R^3)}\Big)\\
&\lesssim \|f_1\|_{L^2(\mathbb R^3)}\|f_2\|_{L^2(\mathbb R^3)}.
\end{split}
\end{equation*}

\medskip

$\circ$\quad{\it Low+Low$\rightarrow$Low}\quad
$|\xi_1|,|\xi_2|\lesssim\max\{15,|\alpha|\}$.\quad Via changing
variables and using the Cauchy-Schwarz inequality we can bound
(\ref{eq:3.2}) with
$$\int  K_{ll}
\left(\int
|f_1(\xi_1,\mu_1,\tau_1)f_2(\xi-\xi_1,\mu-\mu_1,\tau-\tau_1)|^2d\tau_1
d\xi_1 d\mu_1\right)^{\frac{1}{2}} g(\xi,\mu,\tau)d\xi d\mu d\tau,
$$
where
$$
K_{ll}=\frac{|\xi|}{<\sigma>^{\frac{1}{2}-}}\left(\int
\frac{d\tau_1d\xi_1d\mu_1}
{<\tau_1-\omega(\xi_1,\mu_1)>^{1+}<\tau-\tau_1-\omega(\xi-\xi_1,\mu-\mu_1)>^{1+}}\right)^{\frac{1}{2}}.
$$
We need only to control $K_{ll}$ using a constant independent of
$\xi,\mu,\tau$. By (\ref{2.1}) we have
$$
K_{ll}\lesssim\frac{|\xi|} {<\sigma>^{\frac{1}{2}-}}\left(\int
\frac{d\xi_1d\mu_1}
{<\tau-\omega(\xi,\mu)-\omega(\xi-\xi_1,\mu-\mu_1)>^{1+}}\right)^{\frac{1}{2}}.
$$
An elementary computation with the change of variables:
$$
\nu=\tau-\omega(\xi,\mu)-\omega(\xi-\xi_1,\mu-\mu_1)
$$
shows
$$
\Big|\frac{d\nu}{d\mu_1}\Big|\gtrsim|\xi|^{\frac{1}{2}}
|\sigma+\xi\xi_1(\xi-\xi_1)(5\xi^2-5\xi\xi_1+5\xi_1^2-3\alpha)-\nu|^{\frac{1}{2}}
$$
and consequently,
$$
K_{ll}\lesssim\frac{|\xi|^{\frac{3}{4}}}{<\sigma>^{\frac{1}{2}-}}\left(
\int \frac{d\xi_1
d\nu}{<\nu>^{1+}|\sigma+\xi\xi_1(\xi-\xi_1)(5\xi^2-5\xi\xi_1+5\xi_1^2-3\alpha)-\nu|^{\frac{1}{2}}}\right)^{\frac{1}{2}}.
$$
By (\ref{2.2}) we further get
$$
K_{ll}\lesssim\left(\int_{|\xi_1|\lesssim\max\{15,|\alpha|\}}
\frac{d\xi_1}{<\sigma+\xi\xi_1(\xi-\xi_1)
(5\xi^2-5\xi\xi_1+5\xi_1^2-3\alpha)>^{\frac{1}{2}}}\right)^{\frac{1}{2}}\lesssim
1.
$$

\medskip

\noindent$\bullet$\quad{\bf Situation 2 -- High Frequency}\quad
$|\xi_1+\xi_2|\gtrsim\max\{10,|\alpha|\}.$

\medskip

$\circ$\quad{\it High+Low$\rightarrow$High}\quad
$|\xi_2|\lesssim\max\{10,|\alpha|\}\lesssim|\xi|\sim|\xi_1|$.\quad As
above, we apply the Cauchy-Schwarz inequality to bound the integral
in (\ref{eq:3.2}) from above with a constant multiple of
$$\int_{} K_{hl}
\left(\int_{}|f_1(\xi_1,\mu_1,\tau_1)f_2(\xi-\xi_1,\mu-\mu_1,\tau-\tau_1)|^2d\tau_1d\xi_1d\mu_1\right)^{\frac{1}{2}}
g(\xi,\mu,\tau)d\xi d\mu d\tau
$$
where
$$
K_{hl}=\frac{|\xi|}{<\sigma>^{\frac{1}{2}-}} \left(\int_{}\frac{d\tau_1d\xi_1d\mu_1}
{<\tau_1-\omega(\xi_1,\mu_1)>^{1+}<\tau-\tau_1-\omega(\xi-\xi_1,\mu-\mu_1)>^{1+}}\right)^{\frac{1}{2}},
$$
but also we have the following estimate
$$
K_{hl}\lesssim\frac{|\xi|}
{<\sigma>^{\frac{1}{2}-}}\left(\int_{}\frac{d\xi_1d\mu_1}
{<\tau-\omega(\xi,\mu)-\omega(\xi-\xi_1,\mu-\mu_1)>^{1+}}\right)^{\frac{1}{2}}.
$$
Under the change of variables
$$
\kappa=\xi\xi_1(\xi-\xi_1)(5\xi^2-5\xi\xi_1+5\xi_1^2-3\alpha);\quad
\nu=\tau-\omega(\xi,\mu)-\omega(\xi-\xi_1,\mu-\mu_1)
$$
the Jacobian determinant $J$ enjoys
$$
J\lesssim\frac{|\kappa|^{\frac{1}{2}}}{|\xi|^{\frac{7}{2}}|\sigma+\kappa-\nu|^{\frac{1}{2}}(|\xi|^5-2|\kappa|)^{\frac{1}{2}}}.
$$
As a by-product of the last inequality and (\ref{2.2}), we obtain
\begin{eqnarray*}
K_{hl}&\lesssim&\frac{1}{|\xi|^{\frac{3}{4}}<\sigma>^{\frac{1}{2}-}}
\left(\int_{}\frac{|\kappa|^\frac12\, d\kappa d\nu}{|\sigma+\kappa-\nu|^{\frac{1}{2}}(|\xi|^5-2|\kappa|)^{\frac{1}{2}}<\nu>^{1+}}\right)^\frac{1}{2}\\
&\lesssim&\frac{1}{|\xi|^{\frac{3}{4}}<\sigma>^{\frac{1}{2}-}}
\left(\int_{}\frac{|\kappa|^\frac12\,
d\kappa}{<\sigma+\kappa>^{\frac{1}{2}}(|\xi|^5-2|\kappa|)^{\frac{1}{2}}}\right)^\frac{1}{2}.
\end{eqnarray*}
Since $|\xi-\xi_1|\lesssim\max\{10,|\alpha|\}$, we have
$|\kappa|\lesssim |\xi|^4$, whence getting
\begin{equation*}
K_{hl}\lesssim\frac{1}{|\xi|^2<\sigma>^{\frac{1}{2}-}}
\left(\int_{|\kappa|\lesssim|\xi|^4}\frac{d\kappa}{<\sigma+\kappa>^{\frac{1}{2}}}\right)^{\frac{1}{2}}\lesssim1.
\end{equation*}

\medskip

$\circ$\quad {\it High+High$\rightarrow$High}\quad
$|\xi_1|,|\xi_2|\gtrsim \max\{10,|\alpha|\}$.\quad Since
$|\xi_1|\geq|\xi_2|$, we have $|\xi_1|\gtrsim|\xi_1+\xi_2|$. Under
this circumstance, we will deal with two cases in the sequel.

\vspace{0.1cm}

$\diamond$\quad{\it Case (i)}\quad
$\max\{|\sigma|,|\sigma_2|\}\gtrsim |\xi_1|^2$.\quad Decomposing the
integral according to $|\xi_1|\sim 2^m$ where $m=1,2,\cdots$, we can
run the dyadic decomposition:
$$
|\sigma|\sim 2^j,\quad|\sigma_1|\sim 2^{j_1},\quad |\sigma_2|\sim
2^{j_2}\quad\hbox{for}\quad j,j_1,j_2=0,1,2,....
$$
If $|\sigma|\geq|\sigma_2|\geq|\xi_1|^2$, then an application of
(\ref{stri-4}) yields that the integral in (\ref{eq:3.2}) is bounded
from above by a constant multiple of
\begin{equation*}
\begin{split}
&\sum_{m\geq1}\sum_{j\geq2m}\sum_{j_1,j_2\geq0}\Big(2^{\frac{3m}{4}}
2^{-j(\frac{1}{2}-)}2^{-j_1(\frac{1}{2}+)}2^{-j_2(\frac{1}{2}+)}\|\eta_{j}(\sigma)g\|_{L^2(\mathbb R^3)}\\
&\quad\times\big\||-i\partial_x|^{\frac{1}{4}}\big(\eta_m(\xi_1)\eta_{j_1}(\sigma_1)f_1\big)^\vee\big\|_{L^4_TL^4(\mathbb
R^2)}
\big\||-i\partial_x|^{\frac{1}{4}}\big(\eta_{j_2}(\sigma_2)f_2\big)^\vee\big\|_{L^4_TL^4(\mathbb
R^2)}\Big)\\
&\lesssim \sum_{m\geq1}\sum_{j\geq
2m}\sum_{j_1,j_2\geq0}\Big(2^{-j(\frac{1}{2}-)}2^{\frac{3m}{4}}
2^{-j_1[(\frac{1}{2}+)-\frac{1}{2}]}2^{-j_2[(\frac{1}{2}+)-\frac{1}{2}]}\\
&\quad\times \|\eta_{j_1}(\sigma_1)f_1\|_{L^2(\mathbb
R^3)}\|\eta_{j_2}(\sigma_2)f_2\|_{L^2(\mathbb R^3)}
\|\eta_{j}(\sigma)g\|_{L^2(\mathbb R^3)}\Big)\\
&\lesssim \|f_1\|_{L^2(\mathbb R^3)}\|f_2\|_{L^2(\mathbb R^3)}.
\end{split}
\end{equation*}
If $|\sigma_2|\geq|\sigma|\geq|\xi_1|^2$, then a further use of
(\ref{stri-4}) derives that the integral in (\ref{eq:3.2}) is
bounded from above by a constant multiple of
\begin{equation*}
\begin{split}
&\sum_{m\geq1}\sum_{j_2\geq2m,j\geq0}\sum_{j_1\geq0}\Big(2^{\frac{m}{2}}
2^{-j(\frac{1}{2}-)}2^{-j_1(\frac{1}{2}+)}2^{-j(\frac{1}{2}+)}\|\eta_{j_2}(\sigma_2)f_2\|_{L^2(\mathbb R^3)}\\
&\quad\times\big\||-i\partial_x|^{\frac{1}{4}}(\eta_m(\xi_1)\eta_{j_1}(\sigma_1)f_1)^\vee\big\|_{L^4_TL^4(\mathbb
R^2)}\big\||-i\partial_x|^{\frac{1}{4}}(\eta_{j}(\sigma)g)^\vee\big\|_{L^4_TL^4(\mathbb R^2)}\Big)\\
&\lesssim
\sum_{m\geq1}\sum_{j_2\geq2m}\sum_{j_1,j_2\geq0}\Big(2^{-j_2(\frac{1}{4}+)}2^{\frac{m}{2}}
2^{-j_1[(\frac{1}{2}+)-\frac{1}{2}]}2^{-j[(\frac{3}{4}-)-\frac{1}{2}]}\\
&\quad\times \|\eta_{j_1}(\sigma_1)f_1\|_{L^2(\mathbb R^3)}
\|\eta_{j_2}(\sigma_2)f_2\|_{L^2(\mathbb R^3)}
\|\eta_{j}(\sigma)g\|_{L^2(\mathbb R^2)}\Big)\\
&\lesssim \|f_1\|_{L^2(\mathbb R^3)}\|f_2\|_{L^2(\mathbb R^3)}.
\end{split}
\end{equation*}

\vspace{0.1cm}

$\diamond$\quad{\it Case (ii)}\quad
$\max\{|\sigma|,|\sigma_2|\}\lesssim |\xi_1|^2$.\quad In this case,
we need to consider the size of the resonance function even more carefully.
This consideration will be done via splitting the estimate into two
pieces according to the size of resonance function.

\vspace{0.1cm}

$\triangleright$\quad{\it Subcase (i)}\quad
$\max\{|\sigma|,|\sigma_1|,|\sigma_2|\}\gtrsim|\xi_1|^4$.\quad This
means that the resonant interaction does not happen and consequently
$|\sigma_1|\gtrsim|\xi_1|^4$. The dyadic decomposition and
(\ref{stri-4}) are applied to deduce that the integral in
(\ref{eq:3.2}) is bounded from above by a constant multiple of

\begin{equation*}
\begin{split}
&\sum_{m\geq1}\sum_{j_1\geq4m}\sum_{2m\geq
j,j_2\geq0}\Big(2^{\frac34 m}
2^{-j(\frac{1}{2}-)}2^{-j_1(\frac{1}{2}+)}2^{-j_2(\frac{1}{2}+)}\|\eta_m(\xi_1)\eta_{j_1}(\sigma_1)f_1\|_{L^2(\mathbb R^3)}\\
&\quad\times\big\||-i\partial_x|^{\frac{1}{4}}(\eta_{j_2}(\sigma_2)f_2)^\vee\big\|_{L^4_TL^4(\mathbb
R^2)}
\big\||-i\partial_x|^{\frac{1}{4}}(\eta_{j}(\sigma)g)^\vee\big\|_{L^4_TL^4(\mathbb R^2)}\Big)\\
&\lesssim \sum_{m\geq1}\sum_{j_1\geq4m}\sum_{2m\geq
j,j_2\geq0}\Big(2^{\frac{3m}{4}}2^{-j_1(\frac{1}{4}+)}
2^{-j[(\frac{3}{4}-)-\frac{1}{2}]}2^{-j_2[(\frac{1}{2}+)-\frac{1}{2}]}\\
&\quad\times\|\eta_{j_1}(\sigma_1)f_1\|_{L^2(\mathbb R^3)}
\|\eta_{j_2}(\sigma_2)f_2\|_{L^2(\mathbb R^3)}
\|\eta_{j}(\sigma)g\|_{L^2(\mathbb R^3)}\Big)\\
&\lesssim \|f_1\|_{L^2(\mathbb R^3)}\|f_2\|_{L^2(\mathbb R^3)}.
\end{split}
\end{equation*}

\vspace{0.1cm}

$\triangleright$\quad{\it Subcase (ii)}\quad
$\max\{|\sigma|,|\sigma_1|,|\sigma_2|\}\lesssim|\xi_1|^4.$\quad This
means that the resonant interaction does happen. By the definition
of the resonant function we have
$$
\Big|\frac{\mu_1}{\xi_1}-\frac{\mu_2}{\xi_2}\Big|^2>
2^{-1}|\xi_1+\xi_2|^2|5(\xi_1^2+\xi_1\xi_2+\xi_2^2)-3\alpha|.
$$
Let
$$
\theta_1=\tau_1-\omega(\xi_1,\mu_1);\quad
\theta_2=\tau_2-\omega(\xi_2,\mu_2),
$$
and $A_{j,j_1,j_2}$ be the image of the following subset of $A^\ast$
$$
\big\{|\xi_1|\geq|\xi_2|\gtrsim \max\{10,|\alpha|\};\ |\sigma|\sim
2^{j},\ |\sigma_1|\sim 2^{j_1},\ |\sigma_2|\sim 2^{j_2};\
\max\{|\sigma|,|\sigma_1|,|\sigma_2|\}\lesssim|\xi_1|^4\}
$$
under the transformation:
$(\xi_1,\mu_1,\tau_1,\xi_2,\mu_2,\tau_2)\mapsto
(\xi_1,\mu_1,\theta_1,\xi_2,\mu_2,\theta_2)$. If in addition
$$
f_{j_1}=\eta_{j_1}(\sigma_1)f_1(\xi_1,\mu_1,\tau_1);\quad
f_{j_2}=\eta_{j_2}(\sigma_2)f_2(\xi_2,\mu_2,\tau_2),
$$
then the integral in (\ref{eq:3.2}) is controlled from above by a constant multiple of
\begin{equation}\label{eq:3.14a}
\begin{split}
&\sum_{j>0}\sum_{j_1,j_2\geq0}\Big(2^{-j(\frac{1}{2}-)}2^{-j_1(\frac{1}{2}+)}2^{-j_2(\frac{1}{2}+)}\\
&\quad\times\int_{A_{j,j_1,j_2}}\Big[|\xi|g\big(\xi,\mu,\theta_1+\omega(\xi_1,\mu_1)+\theta_2+\omega(\xi_2+\mu_2)\big)
\\
&\quad\quad\times\eta_j\big(\theta_1+\theta_2+\omega(\xi_1,\mu_2)+\omega(\xi_2+\mu_2)-
\omega(\xi_1+\xi_2,\mu_1+\mu_2)\big)
\\
&\quad\quad\times
f_{j_1}\big(\xi_1,\mu_1,\theta_1+\omega(\xi_1,\mu_1)\big)
f_{j_2}\big(\xi_2,\mu_2,\theta_2+\omega(\xi_2,\mu_2)\big)\Big]d\xi_1d\mu_1d\xi_2d\mu_2d\theta_1d\theta_2\Big).
\end{split}
\end{equation}
To get the desired estimate, we are led to dominate the following
sum for each fixed natural number $j$:
\begin{equation}\label{eq:3.14}
\begin{split}
&\sum_{j_1,j_2\geq0}\Big(2^{-j_1(\frac{1}{2}+)}2^{-j_2(\frac{1}{2}+)}\\
&\quad\times\int_{A_{j,j_1,j_2}}\Big[|\xi|g\big(\xi,\mu,\theta_1+\omega(\xi_1,\mu_1)+\theta_2+\omega(\xi_2+\mu_2)\big)
\\
&\quad\quad\times\eta_j\big(\theta_1+\theta_2+\omega(\xi_1,\mu_2)+\omega(\xi_2+\mu_2)-
\omega(\xi_1+\xi_2,\mu_1+\mu_2)\big)
\\
&\quad\quad\times f_{j_1}(\xi_1,\mu_1,\theta_1+\omega(\xi_1,\mu_1))
f_{j_2}(\xi_2,\mu_2,\theta_2+\omega(\xi_2,\mu_2))\Big]d\xi_1d\mu_1d\xi_2d\mu_2d\theta_1d\theta_2\Big).
\end{split}
\end{equation}
This will be accomplished via considering two more settings.

\vspace{0.1cm}

$\star$\quad{\it Subsubcase (i)}\quad
$$
\Big|5(\xi_1^4-\xi_2^4)-3\alpha(\xi_1^2-\xi_2^2)-
\Big[\Big(\frac{\mu_1}{\xi_1}\Big)^2-\Big(\frac{\mu_2}{\xi_2}\Big)^2\Big]\Big|>2^{j}.
$$
Under this circumstance, we change the variables
\begin{equation} \label{eq:3.5}
\left\{ \begin{aligned}
        &u=\xi_1+\xi_2 \\
        &v=\mu_1+\mu_2 \\
        &w=\theta_1+\omega(\xi_1,\mu_1)+\theta_2+\omega(\xi_2+\mu_2)\\
        &\mu_2=\mu_2,
        \end{aligned} \right.
\end{equation}
and then obtain its Jacobian determinant
\begin{equation}
\begin{split}
J_{\mu}&=\left|\begin{array}{cccc}
          1 & 1 & 0 & 0 \\
          \,& \,& \,& \,\\
          0 & 0 & 1 & 1 \\
          \,& \,& \,& \,\\
          5\xi^4_1-3\alpha\xi_1^2-\frac{\mu_1^2}{\xi_1^2} &
5\xi^4_2-3\alpha\xi_2^2-\frac{\mu_2^2}{\xi_2^2} & 2\frac{\mu_1}{\xi_1} & 2\frac{\mu_2}{\xi_2} \\
          \,& \,& \,& \,\\
          0 & 0 & 0 & 1
        \end{array}\right|\\
&=5(\xi_1^4-\xi_2^4)-3\alpha(\xi_1^2-\xi_2^2)-
\Big[\Big(\frac{\mu_1}{\xi_1}\Big)^2-\Big(\frac{\mu_2}{\xi_2}\Big)^2\Big].
\end{split}\end{equation}
Suppose now $A^{(1)}_{j,j_1,j_2}$ is the image of the subset of all
points $(\xi_1,\mu_1,\theta_1,\xi_2,\mu_2,\theta_2)\in
A_{j,j_1,j_2}$ obeying the just-assumed Subsubcase (i) condition
under the transformation (\ref{eq:3.5}). Then it is not hard to
deduce that $|J_\mu|\gtrsim 2^{j}$ and so that the sum in (\ref{eq:3.14}) is
\begin{equation}\label{eq:3.6}
\lesssim\sum_{j_1,j_2\geq0}2^{-j_1(\frac{1}{2}+)}2^{-j_2(\frac{1}{2}+)}\int_{A^{(1)}_{j,j_1,j_2}}
\frac{|u|g(u,v,w)}{|J_\mu|}H(u,v,w,\mu_2,\theta_1,\theta_2)dudvdwd\mu_2d\theta_1d\theta_2,
\end{equation}
where $H(u,v,w,\mu_2,\theta_1,\theta_2)$ is just
$\eta_jf_{j_1}f_{j_2}$ with respect to the transformation
(\ref{eq:3.5}). For the fixed variables:
$\theta_1,\theta_2,\xi_1,\xi_2,\mu_1$, we calculate the set length,
denoted by $\Delta_{\mu_2}$, where the free variable $\mu_2$ can
range. More precisely, if
$$f(\mu)=\theta_1+\theta_2-\frac{\xi_1\xi_2}{(\xi_1+\xi_2)}\left((\xi_1+\xi_2)^2
\Big[5(\xi_1^2+\xi_1\xi_2+\xi_2^2)-3\alpha\Big]-
\Big(\frac{\mu_1}{\xi_1}-\frac{\mu}{\xi_2}\Big)^2\right),
$$
then $|f'(\mu_2)|>|\xi_1|^2\gtrsim|u|^2$, and hence
$\Delta_{\mu_2}\lesssim 2^{j}|u|^{-2}$ follows from

\begin{eqnarray*}\label{eq:3.11}
&&|\theta_1+\theta_2+\omega(\xi_1,\mu_2)+\omega(\xi_2+\mu_2)-
\omega(\xi_1+\xi_2,\mu_1+\mu_2)|\\
&&=\left|\theta_1+\theta_2-\frac{\xi_1\xi_2}{(\xi_1+\xi_2)}\left((\xi_1+\xi_2)^2
\Big[5(\xi_1^2+\xi_1\xi_2+\xi_2^2)-3\alpha\Big]-
\Big(\frac{\mu_1}{\xi_1}-\frac{\mu_2}{\xi_2}\Big)^2\right)\right|\\
&&\sim 2^j.
\end{eqnarray*}
By the Cauchy-Schwarz inequality and the inverse change of variables
we have

\begin{eqnarray*}
&&\int_{A^{(1)}_{j,j_1,j_2}} |u|g(u,v,w)
|J_\mu|^{-1}H(u,v,w,\mu_2,\theta_1,\theta_2)dudvdwd\mu_2d\theta_1d\theta_2\\
&&\lesssim 2^{{\frac{j}{2}}}\int |u|g(u,v,w)
\left(\int |J_\mu|^{-2}H^2(u,v,w,\mu_2,\theta_1,\theta_2)d\mu_2\right)^{\frac{1}{2}}dudvdwd\theta_1d\theta_2\\
&&\lesssim 2^{{\frac{j}{2}}}\|g\|_{L^2(\mathbb R^3)}\int\left(\int |J_\mu|^{-1}H^2(u,v,w,\mu_2,\theta_1,\theta_2)dudvdwd\mu_2\right)^{\frac{1}{2}}d\theta_1d\theta_2\\
&&\lesssim\|g\|_{L^2(\mathbb R^3)}\int
\Big(\int\prod_{i=1,2}f^2_{j_i}(\xi_i,\mu_i,\theta_i+\omega(\xi_i,\mu_i))
d\xi_1d\mu_1d\xi_2d\mu_2\Big)^{\frac{1}{2}}d\theta_1d\theta_2\\
&&\lesssim 2^{\frac{j_1}{2}}2^{\frac{j_2}{2}}\|g\|_{L^2(\mathbb
R^3)}\|f_1\|_{L^2(\mathbb R^3)}\|f_2\|_{L^2(\mathbb R^3)}.
\end{eqnarray*}
It follows from (\ref{eq:3.14}) that the sum in
(\ref{eq:3.14a}) is $\lesssim\|f_1\|_{L^2(\mathbb
R^2)}\|f_2\|_{L^2(\mathbb R^2)}.$

\vspace{0.1cm}

$\star$\quad{\it Subsubcase (ii)}\quad
$$
\Big|5(\xi_1^4-\xi_2^4)-3\alpha(\xi_1^2-\xi_2^2)-
\Big[\Big(\frac{\mu_1}{\xi_1}\Big)^2-\Big(\frac{\mu_2}{\xi_2}\Big)^2\Big]\Big|\leq
2^{j}.
$$
In this setting, the change of variables taken in Subsubcase (i)
does not work because the determinant of the Jacobian may be zero.
So, we cannot help finding a new change of variables. Before doing
this, we notice that the size $|\xi_1|\sim 2^m$ (for $m\geq 0$) can
be used but also the integral in (\ref{eq:3.2}) may be rewritten as

\begin{equation}\label{eq:3.16}
\begin{split}
&\sum_{j_1,m\geq0}\sum_{2m>j,j_2\geq0}\Big(2^{-j(\frac{1}{2}-)}2^{-j_1(\frac{1}{2}+)}2^{-j_2(\frac{1}{2}+)}2^{m}\\
&\quad\times\int_{A_{j,j_1,j_2}}\Big[
g(\xi,\mu,\theta_1+\omega(\xi_1,\mu_1)+\theta_2+\omega(\xi_2+\mu_2))
\\
&\quad\quad\times\eta_j\big(\theta_1+\theta_2+\omega(\xi_1,\mu_2)+\omega(\xi_2+\mu_2)-
\omega(\xi_1+\xi_2,\mu_1+\mu_2)\big)
\\
&\quad\quad\times
f_{m,j_1}\big(\xi_1,\mu_1,\theta_1+\omega(\xi_1,\mu_1)\big)
f_{j_2}\big(\xi_2,\mu_2,\theta_2+\omega(\xi_2,\mu_2)\big)\Big]d\xi_1d\mu_1d\xi_2d\mu_2d\theta_1d\theta_2\Big),
\end{split}
\end{equation}
where
$f_{m,j_1}=\eta_m(\xi_1)\eta_{j_1}(\sigma_1)f_1(\xi_1,\mu_1,\tau_1).$
Now, we choose the following transformation:

\begin{equation} \label{eq:3.7}
\left\{ \begin{aligned}
        &u=\xi_1+\xi_2 \\
        &v=\mu_1+\mu_2 \\
        &w=\theta_1+\omega(\xi_1,\mu_1)+\theta_2+\omega(\xi_2+\mu_2)\\
        &\xi_1=\xi_1,
        \end{aligned} \right.
\end{equation}
and moreover assume that $A^{(2)}_{j,j_1,j_2}$ is the image under
(\ref{eq:3.7}) of the set of those points
$(\xi_1,\mu_1,\theta_1,\xi_2,\mu_2,\theta_2)\in A_{j,j_1,j_2}$
satisfying the just-given Subsubcase (ii) condition. A calculation
yields that the associated Jacobian determinant of the last
transformation (\ref{eq:3.7}) is
\begin{equation}
\begin{split}
J_{\xi}&=\left|\begin{array}{cccc}
          1 & 1 & 0 & 0 \\
          \,& \,& \,& \,\\
          0 & 0 & 1 & 1 \\
          \,& \,& \,& \,\\
          5\xi^4_1-3\alpha\xi_1^2-\frac{\mu_1^2}{\xi_1^2} &
5\xi^4_2-3\alpha\xi_2^2-\frac{\mu_2^2}{\xi_2^2} & 2\frac{\mu_1}{\xi_1} & 2\frac{\mu_2}{\xi_2} \\
          \,& \,& \,& \,\\
          1 & 0 & 0 & 0
        \end{array}\right|\\
&=2\Big(\frac{\mu_1}{\xi_1}-\frac{\mu_2}{\xi_2}\Big).
\end{split}\end{equation}
From this formula it follows that $|J_\xi|\gtrsim|\xi_1|.$ Next, we
fix $\theta_1,\theta_2,\xi_2,\mu_1,\mu_2$, and estimate the interval
length $\Delta_{\xi_1}$ of the free variable $\xi_1$. Putting
\begin{equation}\label{eq:3.18}h(\xi)=5(\xi^4-\xi_2^4)-3\alpha(\xi^2-\xi_2^2)-\Big[\Big(\frac{\mu_1}{\xi}\Big)^2-
\Big(\frac{\mu_2}{\xi_2}\Big)^2\Big],
\end{equation}
we compute
\begin{equation}\label{eq:3.17}
h'(\xi)=20\xi^3-6\alpha\xi+2(\mu_1/\xi)^2\xi^{-1}.
\end{equation}
Since now $h'(\xi_1)$ has the same sign as $\xi_1$'s, we conclude
$|h'(\xi_1)|\gtrsim|\xi_1|^3$, thereby finding
$\Delta_{\xi_1}\lesssim 2^{j-3m}.$ Consequently, the sum in
(\ref{eq:3.16}) is
\begin{equation}\label{eq:3.8}
\lesssim\sum_{j_1,m\geq0}\sum_{2m>j,j_2\geq0}2^{j(-\frac{1}{2}+)}2^{m}\int_{A^{(2)}_{j,j_1,j_2}}
\frac{g(u,v,w)}{|J_\xi|}H(u,v,w,\xi_1,\theta_1,\theta_2)dudvdwd\xi_1d\theta_1d\theta_2,
\end{equation}
where $H(u,v,w,\xi_1,\theta_1,\theta_2)$ equals
$\eta_jf_{m,j_1}f_{j_1}$ under the change of variables
(\ref{eq:3.7}). Note that by the Cauchy-Schwarz inequality
\begin{eqnarray*}
&&\int_{A^{(2)}_{j,j_1,j_2}}\frac{g(u,v,w)}{
|J_\xi|}H(u,v,w,\xi_1,\theta_1,\theta_2)dudvdwd\xi_1d\theta_1d\theta_2\\
&&\lesssim 2^{-\frac{3}{2}m}2^{\frac{j}{2}}\int g(u,v,w)
\left(\int|J_\xi|^{-2}H^2(u,v,w,\xi_1,\theta_1,\theta_2)d\xi_1\right)^{\frac{1}{2}}dudvdwd\theta_1d\theta_2\\
&&\lesssim 2^{-\frac{3}{2}m}2^{\frac{j}{2}}\|g\|_{L^2(\mathbb
R^3)}\int
\left(\int |J_\xi|^{-2}H^2(u,v,w,\xi_1,\theta_1,\theta_2)dudvdwd\xi_1\right)^{\frac{1}{2}}d\theta_1d\theta_2\\
&&\lesssim 2^{-2m}2^{\frac{j}{2}}\|g\|_{L^2(\mathbb R^3)}\int
\left(\int |J_\xi|^{-1}H^2(u,v,w,\xi_1,\theta_1,\theta_2)dudvdwd\xi_1\right)^{\frac{1}{2}}d\theta_1d\theta_2\\
&&\lesssim 2^{-2m}2^{\frac{j}{2}}\|g\|_{L^2(\mathbb R^3)}\int
\Big(\int\prod_{l=1,2}f^2_{l}(\xi_l,\mu_l,\theta_l+\omega(\xi_l,\mu_l))
d\xi_ld\mu_l\Big)^{\frac{1}{2}}d\theta_1d\theta_2\\
&&\lesssim
2^{-2m}2^{\frac{j}{2}}2^{\frac{j_1}{2}}2^{{\frac{j_2}{2}}}\|g\|_{L^2(\mathbb
R^3)}\|f_1\|_{L^2(\mathbb R^3)}\|f_2\|_{L^2(\mathbb R^3)}.
\end{eqnarray*}
Thus the sum in (\ref{eq:3.16}) is
\begin{eqnarray*}
&&\lesssim
\sum_{m,j_1\geq0}\sum_{2m>j_2,j\geq0}\Big(2^{-j(\frac{1}{2}-)}2^{-m}2^{\frac{j}{2}}
2^{-j_1((\frac{1}{2}+)-\frac{1}{2})}2^{j_2((\frac{1}{2}+)-\frac{1}{2})}\\
&&\quad\times\|g\|_{L^2(\mathbb R^3)}
\|f_1\|_{L^2(\mathbb R^3)}\|f_2\|_{L^2(\mathbb R^3)}\Big)\\
&&\lesssim\|f_1\|_{L^2(\mathbb R^3)}\|f_2\|_{L^2(\mathbb R^3)}.
\end{eqnarray*}
\end{proof}

\section{Proof of Theorem \ref{mainth}}

\noindent$\bullet$\quad{\bf Local well-posedness.}\quad Consider the
integral equation associated with (\ref{5KP})
\begin{eqnarray}\label{eq:4.1}
u(t)=\psi(t)S(t)\phi-\frac{\psi_T(t)}{2}
\int_0^tS(t-t')\partial_x\big(u^2(t')\big)dt',
\end{eqnarray}
where $0<T<1$, and $\psi_T(t)$ is the bump function defined in
Section 2. It is clear that a solution to (\ref{eq:4.1}) is a fixed
point of the nonlinear operator
\begin{equation}\label{eq:4.2}
L(u)=\psi(t)S(t)\phi-\frac{\psi_T(t)}{2}\int_0^t
S(t-t')\partial_x\big(u^2(t')\big)dt'.
\end{equation}
Therefore we are required to verify that $L$ is a contractive
mapping from the following closed set to itself
\begin{eqnarray}
B_a=\Big\{u\in X^{s_1,s_2}_b:\quad\|u\|_{X^{s_1,s_2}_b}\leq
a=4c\|\phi\|_{H^{s_1,s_2}(\mathbb R^2)},\quad 2^{-1}<b\Big\}.
\end{eqnarray}
Here and hereafter $c>0$ is a time-free constant and may vary from one line to the other. By Proposition \ref{prop1} and
Theorem \ref{biestimate}, there exists $\sigma
>0$ such that
\begin{eqnarray}\label{eq:4.5}
\|L(u)\|_{X^{s_1,s_2}_b}\leq c\|\phi\|_{H^{s_1,s_2}(\mathbb
R^2)}+cT^\sigma\|u\|_{X^{s_1,s_2}_b}^2.
\end{eqnarray}
Next, since
$\partial_x(u^2)-\partial_x(v^2)=\partial_x[(u-v)(u+v)]$, we
similarly get
\begin{eqnarray}\label{eq:4.3}
\|L(u)-L(v)\|_{X^{s_1,s_2}_b}&\leq &
cT^\sigma\|u-v\|_{X^{s_1,s_2}_b}
\Big(\|u\|_{X^{s_1,s_2}_b}+\|v\|_{X^{s_1,s_2}_b}\Big).
\end{eqnarray}
Choosing $T=T(\|\phi\|_{H^{s_1,s_2}(\mathbb R^2)})$ such that
$8cT^\sigma\|\phi\|_{H^{s_1,s_2}(\mathbb R^2)}<1$, we deduce from
(\ref{eq:4.5}) and (\ref{eq:4.3}) that $L$ is strictly contractive
on the ball $B_a$. Thus there exists a unique solution $u\in
X^{s_1,s_2}_b([-T,T])\subseteq C\big([-T,T];H^{s_1,s_2}(\mathbb
R^2)\big)$ (thanks to $b>1/2$) to the IVP of the fifth order KP-I
equation. The smoothness of the mapping from $H^{s_1,s_2}(\mathbb
R^2)$ to $X^{s_1,s_2}_b([-T,T])$ follows from the fixed point
argument. Because the dispersive function $\omega(\xi,\mu)$ is
singular at $\xi=0$, the requirement
$|\xi|^{-1}\hat{\phi}\in\mathcal{S}'(\mathbb R^2)$ is necessary in
order to have a well defined time derivative of $S(t)\phi$. So, the
argument for the local well-posedness is complete.

\medskip

\noindent$\bullet$\quad{\bf Global well-posedness.}\quad We first
handle the global well-posedness of (\ref{5KP}) in the anisotropic
Sobolev space $H^{s_1,0}(\mathbb R^2)$ with $s_1\geq 0$. Suppose
$\phi\in H^{s_1,0}(\mathbb R^2)$. Then by local well-posedness there
exists a unique solution $u\in C\big([-T,T];H^{s_1,0}(\mathbb
R^2)\big)$ of (\ref{5KP}). We claim that there exists $T$, depending
on $\|\phi\|_{L^2(\mathbb R^2)}$, such that on the interval $[-T,T]$
one has
\begin{equation}\label{eq:4.7}\sup_{|t|\leq
T}\|u(t)\|_{H^{s_1,0}(\mathbb R^2)}\leq c\|\phi\|_{H^{s_1,0}(\mathbb
R^2)}.
\end{equation}
With the help of this
claim and the local well-posedness part of Theorem \ref{mainth} with $u(T)$ and $u(-T)$ being initial values, we can extend the exit time to the positive infinity step by step in that
the exist time $T'$ depends only on
$$
\|u(T)\|_{L^2(\mathbb R^2)}=\|u(-T)\|_{L^2(\mathbb
R^2)}=\|\phi\|_{L^2(\mathbb R^2)}
$$
and
$$
\max\big\{\|u(T)\|_{H^{s_1,0}(\mathbb R^2)},\ \|u(-T)\|_{H^{s_1,0}(\mathbb
R^2)}\big\}\leq c\|\phi\|_{H^{s_1,0}(\mathbb R^2)}.
$$

To check the claim, let $J^{s_1}_x=(I-\partial_x^2)^{s_1/2}$. Then
from the definitions of the anisotropic Sobolev space and the Bourgain
space it follows that
$$
\|J^{s_1}_xu\|_{L^2(\mathbb R^2)}=\|u\|_{H^{s_1,0}(\mathbb R^2)}
\quad\hbox{and}\quad
\|J^{s_1}_xu\|_{X^{0,0}_b}=\|u\|_{X^{s_1,0}_b}.
$$
Letting $J^{s_1}_x$ act on both sides of the integral equation
(\ref{eq:4.1}), we derive
\begin{eqnarray}\label{eq:4.6}
J^{s_1}_xu(t)=\psi(t)S(t)J^{s_1}_x\phi-\frac{
\psi_T(t)}{2}\int_0^tS(t-t')J^{s_1}_x\partial_x\big(u^2(t')\big)dt'.
\end{eqnarray}
By Proposition \ref{prop1}, we have
\begin{equation}\label{Jhomo}
\|\psi S(t)J^{s_1}_x\phi\|_{X^{0,0}_b}\leq
c\|\phi\|_{H^{s_1,0}(\mathbb R^2)},
\end{equation}
as well as
\begin{equation}\label{Jinhomo}
\Big\|\psi_T(t)
\int_0^tS(t-t')J^{s_1}_x\partial_x\big(u^2(t')\big)\,dt'\Big\|_{X^{0,0}_b}\leq
cT^{1-b+b'}\|J^{s_1}_x\partial_x\big(u^2(t')\big)\|_{X^{0,0}_{b'}}.
\end{equation}
A slight modification of the argument for the bilinear estimates
carried out in Section 3 can produce the following bilinear estimate
\begin{equation}\label{Jbilinear}\|J^{s_1}_x\partial_x(u^2)\|_{X^{0,0}_{-\frac{1}{2}+}}\le c
\|u\|_{X^{0,0}_{\frac{1}{2}+}}\|J^{s_1}_xu\|_{X^{0,0}_{\frac{1}{2}+}}.
\end{equation}
Combining (\ref{Jhomo}), (\ref{Jinhomo}) and (\ref{Jbilinear}), we
get
$$\|J^{s_1}_xu\|_{X^{0,0}_b}\leq c\|\phi\|_{H^{s_1,0}}+cT^\sigma\|u\|_{X^{0,0}_b}\|J^{s_1}_xu\|_{X^{0,0}_b}.$$
By (\ref{eq:4.5}) with $s_1=s_2=0$, we can choose
$T=T(\|\phi\|_{L^2(\mathbb R^2)})$ such that
$cT^\sigma\|u\|_{X^{0,0}_b}<\frac{1}{2}.$ Thus by (\ref{Jbilinear}),
we have
$$\|J^{s_1}_x\partial_x(u^2)\|_{X^{0,0}_{b}}\leq c\|\phi\|_{H^{s_1,0}(\mathbb R^2)}+2^{-1}\|J^{s_1}_xu\|_{X^{0,0}_b}.$$
Since $b>\frac{1}{2}$, we obtain the fundamental embedding inequality
$$
\sup_{|t|\leq T}\|u(t)\|_{H^{s_1,0}(\mathbb R^2)}\leq
\|J^{s_1}_xu\|_{X^{0,0}_b}\leq c\|\phi\|_{H^{s_1,0}(\mathbb R^2)},
$$
as well as (\ref{eq:4.7}) which verifies the claim.

Similarly, the operator $J^{s_2}_y=(I-\partial_y^2)^{s_2/2}$ can act on both side of the
integral equation (\ref{eq:4.1}). As a result, we get
\begin{equation}\label{Jbilinear1}\|J^{s_2}_y\partial_x(u^2)\|_{X^{s_1,0}_{-\frac{1}{2}+}}\le c
\|u\|_{X^{s_1,0}_{\frac{1}{2}+}}\|J^{s_2}_yu\|_{X^{s_1,0}_{\frac{1}{2}+}},
\end{equation}
thereby obtaining the following estimate
$$
\|J^{s_2}_y\partial_x(u^2)\|_{X^{s_1,0}_{-\frac{1}{2}+}}\leq
c\|\phi\|_{H^{s_1,s_2}(\mathbb R^2)}+
cT^\sigma\|u\|_{X^{s_1,0}_b}\|J^{s_2}_yu\|_{X^{s_1,0}_b}.
$$
By (\ref{eq:4.7}), we can also choose a time $T$ so that it depends
on $\|\phi\|_{H^{s_1,0}(\mathbb R^2)}$ and obeys
$cT^\sigma\|u\|_{X^{s_1,0}}<\frac{1}{2}.$ Finally, we arrive at
$$
\sup_{|t|\le T}\|u(t)\|_{H^{s_1,s_2}(\mathbb R^2)}=\sup_{|t|\leq
T}\|J^{s_2}_yu(t)\|_{H^{s_1,0}(\mathbb R^2)}\leq
\|J^{s_2}_yu\|_{X^{s_1,s_2}_b}\leq c\|\phi\|_{H^{s_1,s_2}(\mathbb
R^2)}.
$$
Note that the previous constant $c>0$ is time-free. So, as before we
can extend the exist time to infinity step by step, and therefore
finish the proof of the global well-posedness.

\vspace{0.5cm}

\hspace{5mm}
\begin{center}
\end{center}

\end{document}